\documentclass[11pt,a4paper,reqno]{amsart}
\usepackage[english]{babel}
\usepackage[applemac]{inputenc}
\usepackage[T1]{fontenc}
\usepackage{palatino}
\usepackage{cite}
\usepackage{esint}
\usepackage{verbatim}
\usepackage{amsmath}
\usepackage{amssymb}
\usepackage{amsthm}
\usepackage{amsfonts}
\usepackage[pdftex]{graphicx}
\usepackage{mathtools}
\usepackage{enumitem}
\usepackage{mathabx}

\usepackage[colorlinks = true, citecolor = black]{hyperref}
\title[Dorronsoro's theorem]{An integralgeometric approach to \\ Dorronsoro estimates}
\author{Tuomas Orponen}
\address{University of Helsinki, Department of Mathematics and Statistics}
\thanks{T.O. is supported by the Academy of Finland via the project \emph{Quantitative rectifiability in Euclidean and non-Euclidean spaces}, grant No. 309365.}
\email{tuomas.orponen@helsinki.fi}

\newcommand{\R}{\mathbb{R}}

\newcommand{\Z}{\mathbb{Z}}

\newcommand{\calL}{\mathcal{L}}
\newcommand{\calD}{\mathcal{D}}
\newcommand{\calH}{\mathcal{H}}

\newcommand{\E}{\mathbb{E}}

\newcommand{\calA}{\mathcal{A}}

\newcommand{\spa}{\operatorname{span}}
\newcommand{\diam}{\operatorname{diam}}
\newcommand{\card}{\operatorname{card}}

\newcommand{\calV}{\mathcal{V}}

\newcommand{\osc}{\operatorname{osc}}

\numberwithin{equation}{section}

\theoremstyle{plain}
\newtheorem{thm}[equation]{Theorem}

\newtheorem{lemma}[equation]{Lemma}

\newtheorem{ex}[equation]{Example}
\newtheorem{cor}[equation]{Corollary}
\newtheorem{proposition}[equation]{Proposition}

\theoremstyle{definition}

\newtheorem{definition}[equation]{Definition}

\theoremstyle{remark}
\newtheorem{remark}[equation]{Remark}

\newcommand{\nref}[1]{(\hyperref[#1]{#1})}

\begin{document}

\maketitle

\begin{abstract} A theorem of Dorronsoro from 1985 quantifies the fact that a Lipschitz function $f \colon \R^{n} \to \R$ can be approximated by affine functions almost everywhere, and at sufficiently small scales. This paper contains a new, purely geometric, proof of Dorronsoro's theorem. In brief, it reduces the problem in $\R^{n}$ to a problem in $\R^{n - 1}$ via integralgeometric considerations. For the case $n = 1$, a short geometric proof already exists in the literature. 

A similar proof technique applies to parabolic Lipschitz functions $f \colon \R^{n - 1} \times \R \to \R$. A natural Dorronsoro estimate in this class is known, due to Hofmann. The method presented here allows one to reduce the parabolic problem to the Euclidean one, and to obtain an elementary proof also in this setting. As a corollary, I deduce an analogue of Rademacher's theorem for parabolic Lipschitz functions. \end{abstract} 

\tableofcontents

\section{Introduction} Let $f \colon \R^{n} \to \R$ be a Lipschitz function. By Rademacher's theorem, $f$ is differentiable Lebesgue almost everywhere. In particular, $f$ is approximated by affine maps on sufficiently small neighbourhoods of almost every point in $\R^{n}$. Quantifying this statement is a classical problem. Consider the following coefficients:
\begin{displaymath} \beta_{p}(Q) := \inf_{A} \left[\frac{1}{\diam(Q)^{n}} \int_{Q} \left(\frac{|f - A|}{\diam(Q)} \right)^{p} \, d\calL^{n} \right]^{1/p}, \end{displaymath}  
where $Q \subset \R^{n}$ is a bounded set, $1 \leq p < \infty$, and $\inf$ runs over all affine maps $A \colon \R^{n} \to \R$. For $p = \infty$, define
\begin{displaymath} \beta_{\infty}(Q) := \inf_{A} \sup_{x \in Q} \frac{|f(x) - A(x)|}{\diam(Q)}. \end{displaymath}
Assuming that $f$ is $L$-Lipschitz, one has $\beta_{p}(Q) \lesssim L$ for any bounded set $Q \subset \R^{n}$, and $1 \leq p \leq \infty$. Moreover, Rademacher's theorem implies that $\beta_{p}(B(x,r)) \to 0$ as $r \to 0$ for Lebesgue almost every $x \in \R^{n}$. In the 1985, Dorronsoro \cite[Theorem 2]{MR796440} quantified this corollary of Rademacher's theorem in the following way:
\begin{thm}[Dorronsoro's theorem]\label{main} Let $f \colon \R^{n} \to \R$ be an $L$-Lipschitz function. Then, for $1 \leq p < 2n/(n - 2)$ (and $p = \infty$ for $n = 1$), the following holds for any $C \geq 1$:
\begin{displaymath} \mathop{\sum_{Q \in \calD}}_{Q \subset Q_{0}} \beta_{p}(CQ)^{2}|Q| \lesssim_{C} L|Q_{0}|, \qquad Q_{0} \in \calD. \end{displaymath} 
Here $\calD$ is the family of standard dyadic cubes on $\R^{n}$, $|Q| = \ell(Q)^{n}$ stands for the Lebesgue measure of $Q$, and $CQ$ is the cube which is concentric with $Q$ and has $\diam(CQ) = C\diam(Q)$. \end{thm}

This paper contains a new, purely geometric, proof of Theorem \ref{main}. Dorronsoro's theorem has numerous applications in the theory of quantitative rectifiability and singular integrals; here are a few examples. Jones \cite{MR1013815} used the case $n = 1$ to give a proof for the $L^{2}$-boundedness of the Cauchy integral on Lipschitz graphs. A similar approach works in higher dimensions, and for more general singular integrals, as shown by Tolsa \cite{MR2481953}. Dorronsoro's theorem is an important tool in the machinery behind David and Semmes' theory of uniformly rectifiable sets, see \cite[Section 10]{DS1}. Recently, Dorronsoro's theorem was used as a tool in Azzam and Schul's work \cite{MR3770170} on higher dimensional traveling salesman theorems.

\subsection{Existing proof strategies} There are at least three different proofs of Theorem \ref{main} in the literature. Dorronsoro's original argument in \cite{MR796440} is rather indirect. Instead of Lipschitz functions, Dorronsoro formulates his result in terms of functions in $W^{1,p}(\R^{n})$. During the proof, Dorronsoro first establishes an analogue of his result for functions in the fractional order Sobolev spaces $W^{\alpha,p}(\R^{n})$, for $0 < \alpha < 1$ and $1 < \alpha < 2$, and finally derives the case $\alpha = 1$ (the most relevant case for geometric applications) by complex interpolation. This is not an elementary proof, but it also gives much more information than Theorem \ref{main} (which is only a special case of \cite[Theorem 2]{MR796440}).

A much shorter proof is outlined in the appendix of Azzam's paper \cite{MR3512428}, see in particular \cite[p. 645]{MR3512428}. This proof avoids interpolation, but is crucially based on the Fourier transform converting differentiation into multiplication. It is not clear -- at least to the author -- who this approach should be attributed to. Azzam states that the proof is well-known, and that pieces of it can be found in the lectures \cite{MR1104656} of Christ, and unpublished lecture notes of David. 

Most recently, Hyt\"onen and Naor \cite[Theorem 5]{HN} found another proof which avoids complex interpolation. It is based on basic properties of the heat semigroup, Littlewood-Paley theory, martingale arguments, and Rota's representation theorem. The proof of Hyt\"onen and Naor is longer than Azzam's, but perhaps lighter than Dorronsoro's original: it does not require developing the theory of fractional order Sobolev spaces. 

\subsection{Motivation for a geometric proof}\label{motivation} There is one more proof of Dorronsoro's theorem, which was not mentioned above: when $n = 1$, one can infer Theorem \ref{main} from Jones' traveling salesman theorem \cite{MR1069238} applied to the the graph $\Gamma_{f} \subset \R^{2}$ of the Lipschitz function $f \colon \R \to \R$. In addition to Jones' original complex-analytic argument, at least two different geometric proofs of the traveling salesman theorem are available, due to Okikiolu \cite{MR1182488} and Bishop and Peres \cite[Chapter 10]{MR3616046}. Further, in the special case of graphs, Okikiolu's argument simplifies very substantially: to the best of my knowledge, the shortest (published) proof of Theorem \ref{main} for $n = 1$ can be found in the book of Garnett and Marshall, \cite[Chapter X, Lemma 2.4]{MR2450237}. 

This approach was omitted from the previous section for the simple reason that it only works when $n = 1$. In this case, however, especially the proof in \cite{MR2450237} is much shorter and more elementary than the methods which work for all $n \geq 1$. In this paper, an equally elementary proof is given for all $n \geq 1$. It does not attempt to generalise any of the geometric arguments from the case $n = 1$; rather, it reduces the problem in $\R^{n}$ to a problem in $\R^{n - 1}$ by integralgeometric considerations. The heaviest tools are Chebyshev's inequality and Fubini's theorem. The case $n = 2$ is particularly simple, and is given separately in Section \ref{quickProof}.

\subsection{Dorronsoro estimates for parabolic Lipschitz functions} The technique developed in the paper can also be used to give an elementary proof of a Dorronsoro estimate for \emph{parabolic Lipschitz functions}. The result is originally due to Hofmann \cite{MR1330241}, but the proof in \cite{MR1330241} is even less elementary than the existing proofs of Dorronsoro's theorem in $\R^{n}$, see Remark \ref{hofmannCitation} for further commentary. To keep the introduction short, I postpone the definition of parabolic Lipschitz functions to Section \ref{parabolicSection}. Once the correct notation has been set up, the main result in the parabolic setting looks exactly the same as Theorem \ref{main} (for $p = 2$), see Theorem \ref{dorronsoro}.

As a quick application, I close the paper by deriving an analogue of Rademacher's theorem for parabolic Lipschitz functions, see Theorem \ref{rademacher}. As far as I know, this result is new.
 
 \subsection{Acknowledgements} I thank Katrin F\"assler for many useful discussions during the preparation of this paper, and in particular for pointing out the references \cite{2018arXiv180507270D, MR1484857}. I am also thankful to Mart\'i Prats for pointing out the reference \cite{MR2450237}. 
 
 \section{Proofs in Euclidean space}

Let $D(n,p)$ be the statement of Theorem \ref{main} for $n \geq 1$ and $p \in [1,\infty]$, and note that $D(n,q)$ is a stronger statement than $D(n,p)$ for $q \geq p$. Most of Dorronsoro's paper \cite{MR796440} is devoted to establishing $D(n,1)$ for all $n \geq 1$, and then $D(n,p)$ for $1 \leq p < 2n/(n - 2)$ is reduced to $D(n,1)$ in \cite[Section 5]{MR796440}. 

The structure of the proof below can be described as follows:
\begin{displaymath} D(1,\infty) + D(n - 1,2) \quad \Longrightarrow \quad D(n,2), \qquad n \geq 2. \end{displaymath}
In particular, assuming $D(1,\infty)$, one obtains $D(n,p)$ for all $n \geq 1$ and $1 \leq p \leq 2$. For $2 < p < 2n/(n - 2)$, one needs to reduce $D(n,p)$ to $D(n,1)$ as in \cite[Section 5]{MR796440}. 

\subsection{Measures on the affine Grassmannian} Let $0 \leq m \leq n$, and let $\calA_{m}$ be the family of all (affine) $m$-planes in $\R^{n}$. For $B \subset \R^{n}$, we write
\begin{displaymath} \calA_{m}(B) := \{V \in \calA_{m} : B \cap V \neq \emptyset\}. \end{displaymath}
We denote by $\eta_{m}$ the unique translation invariant Borel measure on $\calA_{m}$ with the normalisation
\begin{displaymath} \eta_{m}(\calA_{m}(B(1))) = 1. \end{displaymath} 
Here $B(1) := B(0,1)$ is the open unit ball centred at the origin, and translation invariance means that
\begin{displaymath} \eta_{m}(\{V + x : V \in \calV\}) = \eta_{m}(\calV) \end{displaymath}
for all Borel subsets $\calV \subset \calA_{m}$, and $x \in \R^{n}$. In the cases $m = 0$ and $m = n$ one has $\eta_{0} = c_{n}\calL^{n}$ and $\eta_{n} = \delta_{\R^{n}}$. 

In fact,  only the measures $\eta_{1}$ and $\eta_{n - 1}$ will be needed below, and I record a pair of useful representations for them. If $e \in S^{n - 1}$, let $\rho_{e}$ be the orthogonal projection to the line $\ell_{e} := \spa(e)$, and let $\pi_{e}$ be the orthogonal projection to $V_{e} := e^{\perp}$. Then, for some constants $c_{1},c_{n - 1} > 0$, one has
\begin{equation}\label{eta1} \eta_{1}(\calL) = c_{1}\int_{S^{n - 1}} \calH^{n - 1}(\{v \in V_{e} : \pi_{e}^{-1}\{v\} \in \calL\}) \, d\sigma(e), \qquad \calL \subset \calA_{1}, \end{equation}
and
\begin{equation}\label{etaD} \eta_{n - 1}(\calV) = c_{n - 1}\int_{S^{n - 1}} \calH^{1}(\{t \in \ell_{e} : \rho_{e}^{-1}\{t\} \in \calV\}) \, d\sigma(e), \qquad \calV \subset \calA_{n - 1}, \end{equation}
where $\sigma := \calH^{n - 1}|_{S^{n - 1}}$. The reader can either check that \eqref{eta1}-\eqref{etaD} define translation invariant measures on $\calA_{1}$ and $\calA_{n - 1}$. Alternatively, \eqref{eta1}-\eqref{etaD} can be considered as the definitions of $\eta_{1}$ and $\eta_{n - 1}$. 

\subsection{Integralgeometric $\beta$ numbers} Let $f \colon \R^{n} \to \R$ be a continuous function. For 
\begin{displaymath} 0 \leq m \leq d, \quad 1 \leq p < \infty, \quad V \in \calA_{m}, \quad \text{and} \quad Q \subset \R^{n}, \end{displaymath}
we define
\begin{displaymath} \beta_{p}(Q,V) := \inf_{A} \left[\frac{1}{\diam(Q)^{m}} \int_{Q \cap V} \left( \frac{|f - A|}{\diam(Q)} \right)^{p} \, d\calH^{m} \right]^{1/p}. \end{displaymath} 
Here the $\inf$ runs over all affine maps $\R^{n} \to \R$. The reader is supposed to think "cube" when seeing "$Q$", but the definition applies more generally. Note that the index $m$ on the right hand side is implicitly determined by the dimension of the plane $V$. I also define $\beta_{\infty}(Q,V)$ as expected:
\begin{displaymath} \beta_{\infty}(Q,V) := \inf_{A} \sup_{x \in Q \cap V} \frac{|f(x) - A(x)|}{\diam(Q)}. \end{displaymath}
Finally, for $1 \leq p \leq \infty$ and $1 \leq q < \infty$, and a cube $Q \subset \R^{n}$, define the \emph{integralgeometric $\beta$ number}
\begin{equation}\label{alphapq} \beta^{m}_{p,q}(Q) := \left[ \fint_{\calA_{m}(Q)} \beta_{p}(Q,V)^{q} \, d\eta_{m}(V) \right]^{1/q}. \end{equation}
\begin{ex} Since $\calA_{n} = \{\R^{n}\}$ and $\eta_{n} = \delta_{\R^{n}}$, observe that
\begin{displaymath} \beta^{n}_{p,q}(Q) = \beta_{p}(Q,\R^{n}) = \inf_{A} \left[\frac{1}{\diam(Q)^{n}} \int_{Q} \left(\frac{|f - A|}{\diam(Q)} \right)^{p} \, d\calL^{n} \right]^{1/p} = \beta_{p}(Q) \end{displaymath} 
for any $1 \leq q < \infty$, where $\beta_{p}(C_{0}Q)$ was defined in the first section. The number $\beta^{0}_{p,q}(Q)$, on the other hand, is not terribly useful:
\begin{displaymath} \beta^{0}_{p,q}(Q) = \left[ \fint_{Q} \beta_{p}(Q,\{x\})^{q} \, d\calL^{n}(x) \right]^{1/q} = 0 \end{displaymath} 
for any cube $Q \subset \R^{n}$, since evidently $\beta_{p}(Q,\{x\}) = 0$ for all $x \in \R^{n}$. \end{ex}

\subsection{Estimates for the integralgeometric $\beta$ numbers} As already mentioned in the introduction, there is a simple geometric proof available for the case $n = 1$ of Theorem \ref{main}, namely \cite[Chapter X, Lemma 2.4]{MR2450237} (or \cite{MR1069238, MR1182488, MR3616046}). I emphasise that the proof in \cite{MR2450237} gives the case $n = 1$ of Theorem \ref{main} for $p = \infty$; this is crucial below. The following estimate for the numbers $\beta_{\infty,2}^{1}(Q)$ is an easy corollary:
\begin{lemma}\label{carleson1} Let $f \colon \R^{n} \to \R$ be an $L$-Lipschitz function. Then, for any $C \geq 1$,
\begin{displaymath} \mathop{\sum_{Q \in \calD}}_{Q \subset Q_{0}} \beta^{1}_{\infty,2}(CQ)^{2}|Q| \lesssim_{C} L|Q_{0}|, \qquad Q_{0} \in \calD. \end{displaymath}
\end{lemma}

\begin{proof} Start by observing that
\begin{displaymath} \eta_{1}(\calA_{1}(CQ)) \sim_{C} \diam(Q)^{n - 1}, \qquad Q \in \calD, \end{displaymath} 
as follows easily from \eqref{eta1}.  We then fix $Q_{0} \in \calD$ and make the following estimates:
\begin{align*} \mathop{\sum_{Q \in \calD}}_{Q \subset Q_{0}} \beta_{\infty,2}^{1}(CQ)^{2}|Q| & = \mathop{\sum_{Q \in \calD}}_{Q \subset Q_{0}} \left[ \fint_{\calA_{1}(CQ)} \beta_{\infty}(CQ,V)^{2} \, d\eta_{1}(\ell) \right] |Q|\\
& \lesssim_{C} \int_{\calA_{1}(CQ_{0})} \mathop{\sum_{Q \subset Q_{0}}}_{CQ \cap \ell \neq \emptyset} \beta_{\infty}(CQ,\ell)^{2}\diam(Q) \, d\eta_{1}(\ell)\\
& \lesssim_{C} L\int_{\calA_{1}(CQ_{0})} \diam(Q_{0}) \, d\eta_{1}(\ell) \sim_{C} L|Q_{0}|^{n}, \end{align*} 
applying the case $n = 1$ of Theorem \ref{main} to the restriction of $f$ to $\ell$ when passing to the last line. The proof is complete. \end{proof}

A very similar argument gives the following result for the numbers $\beta_{2,2}^{n - 1}(Q)$:

\begin{lemma}\label{carleson2} Let $f \colon \R^{n} \to \R$ be an $L$-Lipschitz function, and assume that Theorem \ref{main} holds in $\R^{n - 1}$. Then, for any $C \geq 1$,
\begin{displaymath} \mathop{\sum_{Q \in \calD}}_{Q \subset Q_{0}} \beta_{2,2}^{n - 1}(CQ)^{2}|Q| \lesssim_{C} L|Q_{0}|, \qquad Q_{0} \in \calD. \end{displaymath} 
\end{lemma}

\begin{proof} This time, deduce from \eqref{etaD} that
\begin{displaymath} \eta_{n - 1}(\calA_{n - 1}(CQ)) \sim_{C} \diam(Q), \qquad Q \in \calD. \end{displaymath}
Then, fix $Q_{0} \in \calD$, and estimate as follows:
\begin{align*} \mathop{\sum_{Q \in \calD}}_{Q \subset Q_{0}} \beta_{2,2}^{n - 1}(CQ)^{2}|Q| & = \mathop{\sum_{Q \in \calD}}_{Q \subset Q_{0}} \left[ \fint_{\calA_{n - 1}(CQ)} \beta_{2}(CQ,V)^{2} \, d\eta_{n - 1}(V) \right] |Q|\\
& \lesssim_{C} \int_{\calA_{n - 1}(CQ_{0})} \mathop{\sum_{Q \subset Q_{0}}}_{CQ \cap V \neq \emptyset} \beta_{2}(CQ,V)^{2}\diam(Q)^{n - 1} \, d\eta_{n - 1}(V)\\
& \lesssim_{C} L\int_{\calA_{n - 1}(CQ_{0})} \diam(Q)^{n - 1} \, d\eta_{n - 1}(V) \sim_{C} L|Q_{0}|, \end{align*} 
applying Theorem \ref{main} to the restriction of $f$ to $V$ when passing to the last line. The proof is complete. \end{proof}

\subsection{Proof of the main theorem} Assume inductively that $n \geq 2$, and Theorem \ref{main} for $p = 2$ is true in dimension $n - 1$. Fix an $L$-Lipschitz function $f \colon \R^{n} \to \R$. All the $\beta$-coefficients appearing below will be defined relative to this function $f$. Possible dependence on the ambient dimension "$n$" will be suppressed in the $\lesssim$ notation. With Lemmas \ref{carleson1} and \ref{carleson2} in mind, fix $C \geq 1$ and set
\begin{equation}\label{form13} \beta(Q) := \beta_{2,2}^{n - 1}(Q)^{2} + \beta_{\infty,2}^{1}(Q) \end{equation} 
for any cube $Q \subset \R^{d}$. Then, the lemmas imply that
\begin{displaymath} \mathop{\sum_{Q \in \calD}}_{Q \subset Q_{0}} \beta(CQ)^{2}|Q| \lesssim L|Q_{0}|, \qquad Q_{0} \in \calD. \end{displaymath} 
So, Theorem \ref{main} will immediately follow from the estimate
\begin{equation}\label{form1} \beta_{2}(cQ) \lesssim \beta(Q) \end{equation} 
for any cube $Q \subset \R^{n}$, where $c > 0$ is a small constant, depending on $n$. For technical reasons, the proof will literally show that $\beta_{2}(cQ) \lesssim \beta(CQ)$, where $0 < 1 \ll 1 \ll C$ are constants depending on $n$, but this clearly implies \eqref{form1}. Applying scalings and translations, it is also easy to see that it is sufficient to prove \eqref{form1} for the cube $Q = [0,1]^{n}$. This notation will be adopted for the rest of the proof (of \eqref{form1} and Theorem \ref{main}). 

If $\beta(Q)$ is small, then $f$ is close to an affine map $A_{V}$ on $Q \cap V$ for most planes $V \in \calA_{n - 1}(Q)$. The main problem in proving \eqref{form1} is that the maps $A_{V}$ associated to different $V \in \calA_{n - 1}(Q)$ may \emph{a priori} have nothing to do with each other: to prove \eqref{form1}, one needs to construct a "global" affine map $A$ approximating $f$ well inside $cQ$. The only candidates available, however, are the maps $A_{V}$. Lemma \ref{mainLemma} below will establish a weak "compatibility" condition for a (randomly selected) $(n + 1)$-tuple $A_{V_{1}},\ldots,A_{V_{n + 1}}$, which will eventually allow the construction of $A$ in Section \ref{conclusion}.

\subsubsection{Quick proof in the plane}\label{quickProof} The proof of \eqref{form1} is very simple in the case $n = 2$, so I sketch it separately, before giving the general details. This discussion is not needed later, so the reader can also skip ahead until Definition \ref{transDef}.

\begin{figure}[h!]
\begin{center}
\includegraphics[scale = 0.9]{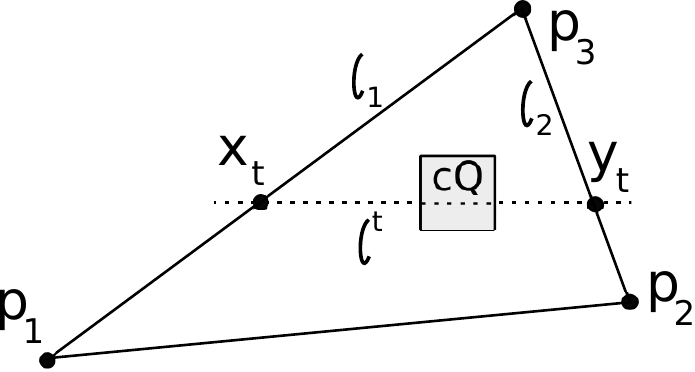}
\caption{The proof in the case $n = 2$.}\label{fig1}
\end{center}
\end{figure}

Choose three lines $\ell_{1},\ell_{2},\ell_{3} \in \calA_{1}(Q)$ at random. Then, if $c > 0$ is sufficiently small, the following things happen (simultaneously) with positive probability:
\begin{itemize} 
\item[(i)] $5cQ \subset \bigtriangleup \subset Q$, where $\bigtriangleup$ is the triangle bounded by $\ell_{1},\ell_{2},\ell_{3}$, see Figure \ref{fig1}.
\item[(ii)] $\beta_{\infty}(Q,\ell_{j}) \lesssim \beta_{\infty,2}^{1}(Q) \leq \beta(Q)$ for $1 \leq j \leq 3$.
\end{itemize} 
Now, let $A_{j}$, $1 \leq j \leq 3$, be an affine map minimising $\beta_{\infty}(Q,\ell_{j}$). Further, let $A$ be the unique affine map which coincides with $f$ at the three corners $p_{1},p_{2},p_{3}$ of $\bigtriangleup$. By (i), these corners are contained in $Q$ and their separation is $\sim_{c} 1$. Further, by (ii),
\begin{equation}\label{form22} |A(p_{i}) - A_{j}(p_{i})| = |f(p_{i}) - A_{j}(p_{i})| \lesssim \beta_{\infty}(Q,\ell_{j}) \lesssim \beta(Q) \end{equation}
whenever $\ell_{j}$ contains $p_{i}$. This implies that $A_{j}$ coincides, up to an error of $\beta(Q)$ with $A$ on (at least) two of the corners of $\bigtriangleup$. It follows that 
\begin{equation}\label{form20} \|A - A_{j}\|_{L^{\infty}(\ell_{j} \cap Q)} \lesssim \beta(Q), \qquad 1 \leq j \leq 3. \end{equation}
To wrap up the proof, write $\ell^{t} := \{(x,y) : y = t\}$. Without loss of generality, assume that 
\begin{equation}\label{form21} \int_{0}^{1} \beta_{\infty}(Q,\ell^{t})^{2} \, dt \lesssim \beta^{1}_{\infty,2}(Q)^{2}. \end{equation}
(By definition of $\beta^{1}_{\infty,2}(Q)$, this is true in almost all directions, even if it happens to fail for horizontal lines). Then, consider a fixed line $\ell^{t} \in \calA_{1}(cQ)$, and let $A^{t}$ be an affine map minimising $\beta_{\infty}(Q,\ell^{t})$. Note that $\ell^{t}$ now meets $\partial \bigtriangleup$ in two points $\{x_{t},y_{t}\}$ with $|x_{t} - y_{t}| \sim_{c} 1$. For notational convenience, assume that always $x_{t} \in \ell_{1}$ and $y_{t} \in \ell_{2}$. Then, 
\begin{displaymath} |A(x_{t}) - A^{t}(x_{t})| \leq |A(x_{t}) - A_{1}(x_{t})| + |A_{1}(x_{t}) - f(x_{t})| + |f(x_{t}) - A^{t}(x_{t})| \lesssim \beta(Q) + \beta_{\infty}(Q,\ell^{t}) \end{displaymath}
by (ii), \eqref{form20} and the definition of $\beta_{\infty}(Q,\ell^{t})$. The same estimate holds for $y_{t}$, and consequently
\begin{displaymath} \|A - A^{t}\|_{L^{\infty}(\ell^{t} \cap Q)} \lesssim_{c} \beta_{\infty}(Q,\ell^{t}) + \beta(Q). \end{displaymath} 
It now follows from Fubini's theorem and \eqref{form21} that
\begin{align*} \beta_{2}(cQ)^{2} \lesssim_{c} \int_{cQ} |f - A|^{2} \, d\calL^{2} \leq & \int_{\{t : \ell^{t} \cap cQ \neq \emptyset\}} \int_{cQ \cap \ell^{t}} |f - A_{t}|^{2} \, d\calH^{1} \, dt\\
& + \int_{\{t : \ell^{t} \cap cQ \neq \emptyset\}} \int_{cQ \cap \ell^{t}} + |A_{t} - A|^{2} \, d\calH^{1} \, dt\\
& \lesssim \int_{0}^{1} \beta_{\infty}(Q,\ell^{t})^{2} + \beta(Q)^{2} \, dt\\
& \lesssim \beta_{\infty,2}^{1}(Q)^{2} + \beta(Q)^{2} \lesssim \beta(Q). \end{align*} 
This completes the proof for $n = 2$.

\subsubsection{The general case}\label{generalCase} Where does the argument above fail $n \geq 3$? The main problem (not the only one) is \eqref{form22}: the lines $\ell_{j}$ should be viewed as $(n - 1)$-planes, so the proof above would yield \eqref{form22} with $\beta_{\infty,2}^{n - 1}(Q)$ on the right hand side. This quantity is no longer dominated by $\beta(Q)$ when $n \geq 3$. Eventually, it is possible to get $\beta_{2,2}^{n - 1}(Q) \leq \beta(Q)$ on the right hand side of \eqref{form22} (see \eqref{form23}), but this takes some more work -- and is, in effect, the main content of Lemma \ref{mainLemma} below.

\begin{definition}\label{transDef} For $\tau > 0$, a family $\calV \subset \calA_{n - 1}$ is called \emph{$\tau$-transversal} if for all $n$-element subsets $\calV_{0} \subset \calV$, the determinant of the normal vectors of the planes in $\calV_{0}$ is at least $\tau$. \end{definition}

\begin{ex}\label{transEx} Let $\bigtriangleup \subset \R^{n}$ be a simplex with $n + 1$ faces and positive Lebesgue measure. Let $V_{1},\ldots,V_{n + 1} \in \calA_{n - 1}$ be the planes containing the faces of $\bigtriangleup$. Then $\{V_{1},\ldots,V_{n + 1}\}$ is $\tau$-transversal for some $\tau > 0$. Also, if $d$ is some natural metric on $\calA_{n - 1}$ (in particular the metric defined below), $\epsilon > 0$ is small enough (depending on $\tau$), and $V_{j}' \in \calA_{n - 1}$ with $d(V_{j},V_{j}') < \epsilon$, then $\{V_{1}',\ldots,V_{n + 1}'\}$ is $(\tau/2)$-transversal. \end{ex}

The following metric $d$ will be used on $\calA_{n - 1}$. Any $V \in \calA_{n - 1}$ can be written as $V = \{x : x \cdot e = t\}$, where $e \in S^{n - 1}$ is normal to $V$, and $t \in \R$. The pair $(e,t)$ is unique up to sign. If $V_{1},V_{2}$ are then associated to $(e_{1},t_{1})$ and $(e_{2},t_{2})$, respectively, write 
\begin{displaymath} d(V_{1},V_{2}) := \min\{|(e_{1},t_{1}) - (e_{2},t_{2})|, |(e_{1},t_{1}) + (e_{2},t_{2})|\}, \end{displaymath}
where $|\cdot|$ refers to Euclidean metric on $S^{n - 1} \times \R \subset \R^{n + 1}$.

\begin{remark}\label{transRem} I record the following corollary of transversality: if $\calV \subset \calA_{n - 1}(Q)$ is a $\tau$-transversal family, and if $C \geq 1$ is sufficiently large (depending on $\tau > 0$), then the unique point in $V_{1} \cap \ldots \cap V_{n}$ lies in $CQ$ for any distinct $V_{1},\ldots,V_{n} \in \calV$. \end{remark}

Recall that $Q = [0,1]^{n}$; this is not too important, but it simplifies notation by eliminating the constant need to normalise by $\diam(Q)$. 

\begin{lemma}\label{mainLemma} Let $\tau > 0$. Fix any $n + 1$ planes $V_{1},\ldots,V_{n + 1} \in \calA_{n - 1}(\tfrac{1}{2}Q)$ so that $\{V_{1},\ldots,V_{n + 1}\}$ is $\tau$-transversal. The following holds if $C \geq 1$ is sufficiently large and $\epsilon > 0$ is sufficiently small, depending on $\tau$. There exist planes $V_{1}',\ldots,V_{n + 1}' \in \calA_{n - 1}(Q)$ with the following properties:
\begin{itemize}
\item[\textup{(a)}] $d(V_{j},V_{j}') \leq \epsilon$,
\item[\textup{(b)}] $\beta_{2}(CQ,V_{j}') \lesssim_{C,\epsilon} \beta_{2,2}^{n - 1}(CQ)$,
\item[\textup{(c)}] For every $1 \leq j \leq n + 1$, there exists an affine quasi-minimiser $A_{j}$ for $\beta_{2}(CQ,V_{j}')$ with the following property. If $1 \leq i_{1} < \ldots < i_{n} \leq n + 1$, and $x$ is the unique point in $V_{i_{1}}' \cap \ldots \cap V_{i_{n}}'$, then $x \in CQ$, and
\begin{displaymath} |f(x) - A_{i_{j}}'(x)| \lesssim_{C,\epsilon} \beta_{2,2}^{n - 1}(CQ), \qquad 1 \leq j \leq d.  \end{displaymath} 
\end{itemize}
\end{lemma}

Property (c) is a "compatibility" condition for the maps $A_{1},\ldots,A_{n + 1}$: it implies that whenever $x$ is a "corner" of the simplex bounded by the planes $V_{1}',\ldots,V_{n + 1}'$, then all the affine maps associated to the $n$ planes meeting at $x$ have nearly the same value at $x$.

\begin{proof}[Proof of Lemma \ref{mainLemma}] The proof is a combination of Chebyshev's inequality and Fubini's theorem. Consider first a fixed plane $V_{j}$, $1 \leq j \leq n + 1$. Let $e_{j} \in S^{n - 1}$ be a direction normal to $V$, and let $e_{j}' \in S^{n - 1}$ with $|e_{j} - e_{j}'| \leq \epsilon$. Write 
\begin{displaymath} V_{j}'(t) := \rho_{e_{j}'}^{-1}\{te_{j}'\} = \{x : x \cdot e_{j}' = t\} \in \calA_{n - 1}, \end{displaymath}
and let $s_{j} \in \R$ be a parameter minimising $t \mapsto d(V_{j},V_{j}'(t))$. Since $V_{j} \cap \tfrac{1}{2}[0,1]^{n} \neq \emptyset$, it is easy to see from the definition of $d$ that
\begin{equation}\label{form17} d(V_{j},V_{j}'(t + s_{j})) \lesssim \epsilon, \qquad 0 \leq t \leq \epsilon. \end{equation} 
Also, $V_{j}'(t) \cap Q \neq \emptyset$ for all $t \in [0,\epsilon]$, if $\epsilon > 0$ is sufficiently small. Then, by definition of $\beta_{2,2}^{n - 1}(CQ)$ and the measure $\eta_{n - 1}$ (recall \eqref{etaD}), and Chebyshev's inequality, one can find $e_{j}' \in S^{n - 1}$ with $|e_{j} - e_{j}'| \leq \epsilon$ such that
\begin{equation}\label{form5} \int_{0}^{\epsilon} \beta_{2}(CQ,V_{j}'(t + s_{j}))^{2} \, dt \lesssim_{C,\epsilon} \beta_{2,2}^{n - 1}(CQ)^{2}. \end{equation} 
Now, applying Chebyshev's inequality again to \eqref{form5}, and recalling \eqref{form17}, one could easily find $V_{j}' := V_{j}'(t + s_{j})$, $0 \leq t \leq \epsilon$, satisfying (a) and (b). For (c), more work is needed.

For $t \in [0,\epsilon]$, let $A_{j,t}$ be an affine (quasi-)minimiser for $\beta_{2}(CQ,V_{j}'(t + s_{j}))$. In other words,
\begin{equation}\label{form6} \int_{CQ \cap V_{j}'(t + s_{j})} |f(x) - A_{j,t}(x)|^{2} \, d\calH^{n - 1}(x) \lesssim \beta_{2}(CQ,V_{j}'(t + s_{j}))^{2}. \end{equation} 
As $t \in [0,\epsilon]$ varies, the sets $CQ \cap V_{j}'(t + s_{j})$ foliate a certain part of $CQ$, which I denote by $B_{j}$, that is,
\begin{displaymath} B_{j} := CQ \cap \bigcup_{t \in [0,\epsilon]} V_{j}'(t + s_{j}). \end{displaymath}
Combining \eqref{form5}-\eqref{form6} with yet another application of Chebyshev's inequality, there is a subset $G_{j} \subset B_{j}$ of measure $|G_{j}| \geq (1 - \epsilon^{2d})|B_{j}|$ such that
\begin{equation}\label{form7} |f(x) - A_{j,t}(x)| \lesssim_{C,\epsilon} \beta_{2,2}^{n - 1}(CQ), \qquad x \in G_{j}, \end{equation} 
where $t \in [0,\epsilon]$ is the unique parameter such that $x \in V_{j}'(t + s_{j})$.

\begin{remark}[Mid-proof remark]\label{midProofRem} Let $1 \leq i_{1} < \ldots < i_{n} \leq n + 1$, and let 
\begin{displaymath} V_{i_{1}}'(t_{1} + s_{i_{1}}),\ldots,V_{i_{n}}'(t_{n} + s_{i_{n}}) \in \calA_{n - 1}(Q) \end{displaymath}
be planes with $t_{j} \in [0,\epsilon]$, with the notation above. Then, if $\epsilon > 0$ is small enough (depending on $\tau$), and $C \geq 1$ is large enough (depending on $\epsilon)$, and recalling \eqref{form17}, Example \ref{transEx}, and Remark \ref{transRem}, the intersection $V_{i_{1}}'(t_{1} + s_{i_{1}}) \cap \ldots \cap V_{i_{n}}'(t_{n} + s_{i_{n}})$ contains a single point which lies in $CQ$.
\end{remark}

Now, fix $1 \leq i_{1} < \ldots < i_{n} \leq n + 1$.  Observe that, assuming $\epsilon > 0$ so small $\det(e_{i_{1}}',\ldots,e_{i_{n}}') \geq \tau/2$, the map 
\begin{displaymath} \Phi = \Phi_{i_{1},\ldots,i_{n}} \colon [0,\epsilon]^{n} \to B_{i_{1}} \cap \ldots \cap B_{i_{n}}, \end{displaymath}
defined by the relation
\begin{displaymath} \Phi(t_{1},\ldots,t_{n}) \in V_{i_{1}}'(t_{1} + s_{i_{1}}) \cap \ldots \cap V_{i_{n}}'(t_{n} + s_{i_{n}}) \end{displaymath} 
has determinant $\gtrsim_{\tau} 1$. Here we used Remark \ref{midProofRem} to ensure that the target of $\Phi$ actually lies in $B_{i_{1}} \cap \ldots \cap B_{i_{n}}$. It follows that $\Phi^{-1}$ is well-defined on the range of $\Phi$, and recalling the definition of $G_{j}$, one has
\begin{displaymath} |\Phi^{-1}([B_{i_{1}} \cap \ldots \cap B_{i_{n}}] \setminus [G_{i_{1}} \cap \ldots \cap G_{i_{n}}])| \lesssim_{\tau} n\epsilon^{2n}. \end{displaymath}
Since $|[0,\epsilon]^{n}| = \epsilon^{n}$, this implies that for $\epsilon > 0$ small enough, a random choice $(t_{1},\ldots,t_{n}) \in [0,\epsilon]^{n}$ satisfies
\begin{equation}\label{form8} V_{i_{1}}'(t_{1} + s_{i_{1}}) \cap \ldots \cap V_{i_{n}}'(t_{n} + s_{i_{n}}) = \{\Phi(t_{1},\ldots,t_{n})\} \subset G_{i_{1}} \cap \ldots \cap G_{i_{n}} \end{equation}
and
\begin{equation}\label{form100} \beta_{2}(CQ,V_{j}'(t_{j} + s_{i_{j}})) \lesssim_{C,\epsilon} \beta^{n - 1}_{2,2}(CQ), \qquad 1 \leq j \leq n \end{equation}
with probability $\geq 1 - \epsilon$. The assertion about \eqref{form100} simply follows from \eqref{form5}. Unwrapping the definitions, \eqref{form8} means that the unique point $x \in V_{i_{1}}'(t_{1} + s_{i_{1}}) \cap \ldots \cap V_{i_{n}}(t_{n} + s_{i_{n}})$ lies in each $G_{i_{j}}$, $1 \leq j \leq n$, which by \eqref{form7} means that 
\begin{displaymath} |f(x) - A_{i_{j},t_{j}}(x)| \lesssim_{C,\epsilon} \beta_{2,2}^{n - 1}(CQ), \qquad 1 \leq j \leq n, \end{displaymath}
as required in (c). However, a little technicality remains: the choice of $(t_{1},\ldots,t_{n})$ depends on the choice of $1 \leq i_{1} < \ldots < i_{n} \leq n + 1$, because $\Phi = \Phi_{i_{1},\ldots,i_{n}}$ does. A further combinatorial argument is therefore required.

Recall again that a randomly selected $n$-tuple $(t_{1},\ldots,t_{n}) \in [0,\epsilon]^{n}$ satisfied \eqref{form8}-\eqref{form100} with probability $\geq 1 - \epsilon$, for a fixed choice of $1 \leq i_{1} < \ldots < i_{n} \leq n + 1$. Denote the corresponding set of (good) $n$-tuples by $H_{i_{1},\ldots,i_{n}}$, so that 
\begin{equation}\label{form10} |H_{i_{1},\ldots,i_{n}}| \geq (1 - \epsilon)\epsilon^{n}. \end{equation}
Consider next the map $\Psi \colon [0,\epsilon]^{n + 1} \to (\R^{n})^{n}$,
\begin{displaymath} \Psi(t_{1},\ldots,t_{n + 1}) := (\Psi_{1}(t_{1},\ldots,t_{n + 1}),\ldots,\Psi_{n}(t_{1},\ldots,t_{n + 1})), \end{displaymath}
where
\begin{displaymath} \Psi_{j}(t_{1},\ldots,t_{n + 1}) = \Phi_{1,\ldots,j - 1,j + 1,\ldots,n + 1}(t_{1},\ldots,t_{j - 1},t_{j + 1},\ldots,t_{n + 1}). \end{displaymath}
When $1 \leq j \leq n + 1$ is fixed, then \eqref{form8}-\eqref{form100} imply that
\begin{equation}\label{form8b} \Psi_{j}(t_{1},\ldots,t_{n + 1}) \in G_{1} \cap \ldots \cap G_{j - 1} \cap G_{j + 1} \cap \ldots \cap G_{n + 1} \end{equation} 
and
\begin{equation}\label{form10b} \beta_{2}(CQ,V_{i}'(t_{i} + s_{i})) \lesssim_{C,\epsilon} \beta^{n - 1}_{2,2}(CQ), \qquad i \in \{1,\ldots,j - 1,j + 1,\ldots,n + 1\} \end{equation}
for all $(t_{1},\ldots,t_{n + 1}) \in [0,\epsilon]^{n}$ with
\begin{equation}\label{form9} (t_{1},\ldots,t_{j - 1},t_{j - 1},\ldots,t_{n + 1}) \in H_{1,\ldots,j - 1,j + 1,\ldots,n + 1} \quad \text{and} \quad t_{j} \in [0,\epsilon], \end{equation}
because neither condition \eqref{form8b}-\eqref{form10b} depends on the variable $t_{j}$. Denote the set of points in $[0,\epsilon]^{n + 1}$ satisfying \eqref{form9} by $\tilde{H}_{j}$, so $|\tilde{H}_{j}| \geq (1 - \epsilon)\epsilon^{n + 1}$ by \eqref{form10}. Now, if $0 < \epsilon < (n + 1)^{-1}$ (in addition to all previous requirements), there exists
\begin{displaymath} (t_{1},\ldots,t_{n + 1}) \in \bigcap_{j = 1}^{n + 1} \tilde{H}_{j}. \end{displaymath}
This implies that 
\begin{equation}\label{form11} \beta_{2}(CQ,V_{i}'(t_{j} + s_{j})) \lesssim_{C,\epsilon} \beta^{n - 1}_{2,2}(CQ), \qquad 1 \leq j \leq n + 1, \end{equation}
and
\begin{equation}\label{form12} \Phi_{i_{1},\ldots,i_{n}}(t_{i_{1}},\ldots,t_{i_{n}}) \in G_{i_{1}} \cap \ldots \cap G_{i_{n}} \end{equation}
simultaneously for all $1 \leq i_{1} < \ldots < i_{n} \leq n + 1$. Now, given the vector $(t_{1},\ldots,t_{n + 1})$ as above, one can finally define
\begin{displaymath} V_{j}' := V_{j}'(t_{j} + s_{j}), \qquad 1 \leq j \leq n + 1. \end{displaymath}
Let $A_{j} := A_{j,t_{j}}$ be the minimiser for $\beta_{2}(Q,V_{j}')$, as in \eqref{form7}. Then (b) holds by \eqref{form11} and (c) holds by \eqref{form12} (repeating the argument given after \eqref{form100}). \end{proof}

\subsubsection{Concluding the proof of \eqref{form1}}\label{conclusion} All the pieces are now in place to conclude the proof of \eqref{form1}, and hence the proof of Theorem \ref{main}. The remaining arguments will largely follow those already seen in Section \ref{quickProof}.
\begin{proof}[Proof of \eqref{form1}] Let $C \geq 1$ be a constant, which is at least as large as the constant from the preceding lemma, and let $0 < c \ll 1$ be the constant from \eqref{form1}; I will tacitly assume that $c$ is as small as needed for the following arguments. Let $\bigtriangleup_{0}$ be a (closed) simplex bounded by $n + 1$ faces (of dimension $n - 1$) with 
\begin{displaymath} 10cQ \subset \bigtriangleup_{0} \subset \tfrac{1}{2}Q. \end{displaymath}
Let $V_{1},\ldots,V_{n + 1} \in \calA_{n - 1}(\tfrac{1}{2}Q)$ be the planes containing the faces of $\bigtriangleup_{0}$. Then, apply Lemma \ref{mainLemma} to locate the planes $V_{1}',\ldots,V_{n + 1}' \in \calA_{n - 1}(Q)$ satisfying (a)-(c), and let $A_{j}$ be the minimiser for $\beta_{2}(CQ,V_{j}')$ from (c). Finally, let $\bigtriangleup$ be the simplex whose $(n - 1)$-faces are contained on the planes $V_{1}',\ldots,V_{n + 1}'$; denote these faces by $\bigtriangleup_{j}$. Choosing $\epsilon > 0$ small enough in (a), it still holds that $5cQ \subset \bigtriangleup$. The dependence on this "$\epsilon$" will be suppressed form the $\lesssim$ notation, as it is a constant depending only on $n$.

Let $A$ be the unique affine map whose values coincide with $f$ at the $n + 1$ corners of $\bigtriangleup$. I claim that $A$ is close to $A_{j}$ on $\bigtriangleup_{j}$. To see this, note that $\bigtriangleup_{j}$ is a simplex containing $n$ corners $x^{j}_{1},\ldots,x^{j}_{n}$ of $\bigtriangleup$ (in particular, the separation of these points is $\sim_{c} 1$), and 
\begin{equation}\label{form23} |A(x^{j}_{i}) - A_{j}(x^{j}_{i})| = |f(x^{j}_{i}) - A_{j}(x^{j}_{i})| \lesssim_{C} \beta_{2,2}^{n - 1}(CQ), \qquad 1 \leq i \leq d, \end{equation}
by the choice of $A$, and (c) of Lemma \ref{mainLemma}. It follows that
\begin{displaymath} \|A - A_{j}\|_{L^{\infty}(\bigtriangleup_{j})} \lesssim_{C} \beta^{n - 1}_{2,2}(CQ). \end{displaymath}
This estimate, combined with (b) of Lemma \ref{mainLemma}, implies that $f$ is well-approximated by $A$ on the sides of $\bigtriangleup$. It remains to argue that $f$ is, also, well-approximated by $A$ inside $\bigtriangleup$ -- and, in particular, on $cQ$.

The argument is similar to the one seen in Section \ref{quickProof}. Recall the definition of $\beta_{\infty,2}^{1}(CQ)$ from \eqref{alphapq} and the definition of the measure $\eta_{1}$ from \eqref{eta1}. One can now find $e \in S^{n - 1}$ such that the following holds. Let $\pi := \pi_{e}$ be the orthogonal projection to $V := e^{\perp}$, and write $\ell_{v} := \pi^{-1}\{v\}$ for $v \in V$. Then, in analogy with \eqref{form21},
\begin{equation}\label{form16} \int_{\pi(cQ)} \beta_{\infty}(CQ,\ell_{v})^{2} \, d\calH^{n - 1}(v) \lesssim_{C} \beta_{\infty,2}^{1}(CQ)^{2}. \end{equation} 
For technical reasons, I mention here that the vector $e \in S^{1}$ can be chosen arbitrarily close to any given vector $e_{0} \in S^{1}$, allowing the implicit constants in \eqref{form16} to depend on $|e - e_{0}|$.

Note that the lines $\ell_{v}$, $v \in \pi(cQ)$, meet $cQ$, and hence also $\partial \bigtriangleup$. In fact, $\partial \bigtriangleup \cap \ell_{v} = \{x_{v},y_{v}\}$ with $|x_{v} - y_{v}| \sim_{c} 1$ whenever $v \in \pi(cQ)$, because $cQ \subset 5cQ \subset \bigtriangleup$. Write
\begin{equation}\label{deltav} \delta(v) := \max\{|f(x_{v}) - A(x_{v})|, |f(y_{v}) - A(y_{v})|\}. \end{equation}
Then, for $v \in \pi(cQ)$ fixed, let $A^{v}$ be an affine map minimising for $\beta_{\infty}(CQ,\ell_{v})$, and note that
\begin{align*} |A(x_{v}) - A^{v}(x_{v})| & \leq |A(x_{v}) - f(x_{v})| + |f(x_{v}) - A^{v}(x_{v})| \lesssim \delta(v) + \beta_{\infty}(CQ,\ell_{v}), \end{align*}
using also that $\bigtriangleup \subset CQ$. Since the same estimate holds for $y_{v}$, and $|x_{v} - y_{y}| \sim_{c} 1$, one infers that
\begin{displaymath} \|f - A\|_{L^{\infty}(cQ \cap \ell_{v})} \leq \beta_{\infty}(CQ,\ell_{v}) +  \|A - A^{v}\|_{L^{\infty}(\bigtriangleup \cap \ell_{v})} \lesssim_{c} \delta(v) + \beta_{\infty}(CQ,\ell_{v}). \end{displaymath}
Consequently, by Fubini's theorem,
\begin{align*} \beta_{2}(cQ)^{2} \lesssim_{c} \int_{cQ} |f - A|^{2} \, d\calL^{n} & = \int_{\pi(cQ)} \int_{\ell_{v} \cap cQ} |f - A|^{2} \, d\calH^{1} \, d\calH^{n - 1}(v)\\
& \lesssim_{c} \int_{\pi(cQ)} [\beta_{\infty}(CQ,\ell_{v})^{2} + \delta(v)^{2}] \, d\calH^{n - 1}(v). \end{align*}
The integral of the numbers $\beta_{\infty}(CQ,\ell_{v})^{2}$ is controlled by $\beta_{\infty,2}^{1}(CQ)^{2} \leq \beta(CQ)^{2}$ by \eqref{form16}. To deal with the numbers $\delta(v)$, one should start by recalling their definition from \eqref{deltav}. Since $x_{v},y_{v} \in \partial \bigtriangleup$, both points are contained on sets of the form $CQ \cap V_{j}'$ for some $1 \leq j \leq n + 1$, and hence the $L^{2}$-integral of $v \mapsto \delta(v)$ is bounded as follows (recalling also (b) of Lemma \ref{mainLemma}):
\begin{displaymath} \int_{\pi(cQ)} \delta(v)^{2} \, d\calH^{n - 1}(v) \lesssim \sum_{j = 1}^{n + 1} \beta_{2}(CQ,V_{j}')^{2} \lesssim_{C} \beta_{2,2}^{n - 1}(CQ)^{2} \leq \beta(CQ)^{2}. \end{displaymath} 
To be precise, the first estimate requires that the lines $\ell_{v}$, $v \in V$, are reasonably transversal to each plane $V_{j}'$, but this can be arranged by choosing the vector $e \in S^{n - 1}$ appropriately; recall the remark after \eqref{form16}. Combining the estimates above, one has $\beta_{2}(cQ) \lesssim_{c,C} \beta(Q)$, and the proof of \eqref{form1} is complete.  \end{proof}
 
 \section{Proofs in parabolic space}\label{parabolicSection}
 
 \subsection{Parabolic Lipschitz functions} I start with the proper definition of parabolic Lipschitz functions, following \cite{FR, HL}, although, in practice, only geometric corollaries of it will be used.
\begin{definition}\label{parabolicLipschitz} A continuous function $\psi \colon \R^{n} \to \R$ is called \emph{parabolic $L$-Lipschitz} if 
\begin{equation}\label{horizontalLipschitz} |\psi(x,t) - \psi(y,t)| \leq L|x - y|, \qquad x,y \in \R^{n - 1}, \: t \in \R, \end{equation} 
and $\mathbb{D}_{n}\psi \in \mathrm{BMO}(\R^{n})$ with norm at most $L$, where $\mathbb{D}_{n}$ stands for the half-order time derivative defined by
\begin{displaymath} \mathbb{D}_{n}\psi(x,t) = \left(\frac{\tau}{\|(\xi,\tau)\|} \hat{A}(\xi,\tau) \right)^{\vee}(x,t), \qquad (x,t),(\xi,\tau) \in \R^{n - 1} \times \R, \end{displaymath}
and $\mathrm{BMO}(\R^{n})$ refers to the BMO space defined through parabolic balls.  \end{definition}
Since this definition is not directly used in the paper, I refer to \cite{FR, HL} for further details. The following proposition is implication $(3) \Longrightarrow (5b)$ (or alternatively $(3) \Longrightarrow (6b)$) in \cite[Theorem 2.3]{2018arXiv180507270D}:
\begin{proposition} Assume that $\psi \colon \R^{n} \to \R$ is parabolic $L$-Lipschitz. Then, there exists a constant $C_{L} > 0$ such that
\begin{equation}\label{DtCarleson} \frac{1}{|Q|} \int_{I_{1}} \iint_{I_{2} \times I_{2}} \frac{|\psi(x,t) - \psi(x,s)|^{2}}{|s - t|^{2}} \, ds \, dt \leq C_{L} \end{equation}
for all parabolic boxes $Q = I_{1} \times I_{2} \subset \R^{n - 1} \times \R$.
\end{proposition} 

A \emph{parabolic box} is any rectangle of the form $I_{1} \times I_{2} \subset \R^{n}$, where $I_{1} \subset \R^{n - 1}$ is a standard Euclidean cube of side-length $\ell(I_{1})$, and $I_{2} \subset \R$ is an interval of length $\ell(I_{2}) = |I_{2}| = \ell(I_{1})^{2}$. In the sequel, the space $\R^{n - 1} \times \R$ is equipped with the metric $d$,
\begin{displaymath} d((x,s),(y,t)) := |x - y| + |s - t|^{1/2}. \end{displaymath}
Metric concepts in the parabolic space $\R^{n - 1} \times \R$, such as balls and diameter, are defined using the metric $d$. It is proven in \cite{MR1330241} (or see \cite{MR1727344} for a more easily accessible reference) that parabolic Lipschitz functions are Lipschitz continuous with respect to the metric $d$:
\begin{equation}\label{lipschitz} |\psi(p) - \psi(q)| \lesssim d(p,q), \qquad p,q \in \R^{n}. \end{equation}
In particular, it follows that $\psi$ is $\tfrac{1}{2}$-H\"older in the $t$-variable. However, \eqref{lipschitz} is not a characterisation of parabolic Lipschitz functions: \eqref{DtCarleson} implies that $\psi$ is "often" a little better than $\tfrac{1}{2}$-H\"older in the $t$-variable. 

\begin{remark} The only information used about parabolic Lipschitz functions in this paper are the conditions \eqref{DtCarleson} and \eqref{lipschitz}. To be completely accurate, the Dorronsoro estimate will require \eqref{DtCarleson} plus the "horizontal" Lipschitz condition \eqref{horizontalLipschitz}; the "full" Lipschitz condition \eqref{lipschitz} will only be applied during the proof of the parabolic Rademacher theorem, Theorem \ref{rademacher}. The arguments below will be completely non-Fourier-analytic. Having said that, since the very definition of being parabolic Lipschitz contains the Fourier transform, literally nothing can be proven about these functions without some reference to Fourier analysis. The derivation of \eqref{DtCarleson} from Definition \ref{parabolicLipschitz} in \cite{2018arXiv180507270D} consists of verifying that the \emph{parabolic derivative} $\mathbb{D}\psi$ of $\psi$ is in BMO (as explained on \cite[p. 5-6]{2018arXiv180507270D}, this not quite the same object as $\mathbb{D}_{n}\psi$), and then inferring condition \eqref{DtCarleson} from the classical work of Strichartz \cite{MR578205} (as detailed in the appendix of \cite{2018arXiv180507270D}, more precisely \cite[(8.4)-(8.5)]{2018arXiv180507270D}). \end{remark}

To formulate Dorronsoro's theorem for parabolic Lipschitz functions, let $Q = I_{1} \times I_{2} \subset \R^{n - 1} \times \R$ be a parabolic box, let $\psi \colon \R^{n} \to \R$ be continuous, and let

\begin{equation}\label{AffCoeff} \beta_{p}(Q) := \inf_{A} \left( \frac{1}{\diam(Q)^{n + 1}} \int_{Q} \left(\frac{|\psi(x,t) - A(x)|^{p}}{\diam(Q)} \right)^{p} \, dx \, dt \right)^{1/p}, \quad 1 \leq p < \infty, \end{equation}
where the $\inf$ ranges over affine maps $A \colon \R^{n - 1} \to \R$. These are the parabolic analogues of the numbers $\beta_{p}(Q)$ discussed earlier in Euclidean space, and I make no distinction in the notation; only the parabolic coefficients will appear in the remainder of this document. For technical reasons, I also define an "$L$-restricted" version of the coefficient $\beta_{p}(Q)$, denoted $\beta^{L}_{p}(Q)$, where the "$\inf_{A}$" is only taken over affine maps $A \colon \R^{n - 1} \to \R$ of (Euclidean) Lipschitz constant at most $L \geq 1$. 

While the definition \eqref{AffCoeff} makes sense for any parabolic box $Q$, I will only use it for \emph{dyadic parabolic boxes}. More precisely, let $\calD^{n - 1}_{j}$ and $\calD^{1}_{j}$, $j \in \Z$, stand for the standard families of dyadic cubes and intervals in $\R^{n - 1}$ and $\R$, respectively, of side-length $2^{-j}$. Define $\calD_{j}$, $j \in \Z$, and $\calD$ to be the following families of parabolic boxes:
\begin{displaymath} \calD_{j} := \{I_{1} \times I_{2} : I_{1} \in \calD^{n - 1}_{j} \text{ and } I_{2} \in \calD^{1}_{2j}\} \quad \text{and} \quad \calD := \bigcup_{j \in \Z} \calD_{j}. \end{displaymath}

Here is the parabolic analogue of Dorronsoro's theorem:
\begin{thm}\label{dorronsoro} Let $\psi \colon \R^{n} \to \R$ be a parabolic Lipschitz function (more precisely, a function satisfying \eqref{horizontalLipschitz} and \eqref{DtCarleson}). Then, for any $C \geq 1$,
\begin{displaymath} \mathop{\sum_{Q \in \calD}}_{Q \subset Q_{0}} \beta_{2}(CQ)^{2}|Q| \lesssim_{C} |Q_{0}|, \qquad Q_{0} \in \calD. \end{displaymath}
Here $CQ = CI_{1} \times CI_{2}$ if $Q = I_{1} \times I_{2}$ is a parabolic box. Moreover, if $\psi$ is parabolic $L$-Lipschitz, then $\beta_{2}(CQ)$ above can be replaced by $\beta_{2}^{L}(CQ)$. \end{thm}

\begin{remark}\label{hofmannCitation} Theorem \ref{dorronsoro} is originally due to Hofmann \cite{MR1484857}, see in particular \cite[(35)]{MR1484857}. The proof given after \cite[(35)]{MR1484857} is remarkably short, but it is Fourier-analytic and, more importantly, based on a "localization lemma", see \cite[Appendix, Lemma 2]{MR1484857}. The proof of the localization lemma is not elementary, as it assumes the $L^{2}$-boundedness of first order parabolic Calder\'on commutators of the form $\left[\sqrt{\bigtriangleup - \tfrac{\partial}{\partial_{t}}}, \psi\right]$, established earlier by Hofmann in \cite{MR1330241}. \end{remark}

\subsubsection{Rademacher's theorem for parabolic Lipschitz functions} Since a parabolic Lipschitz map $\psi \colon \R^{n} \to \R$ is Euclidean Lipschitz in the horizontal variables, one can apply Rademacher's theorem to the following end: for Lebesgue almost every $p = (x,t) \in \R^{n - 1} \times \R$, there exists a linear map $A_{p} \colon \R^{n - 1} \to \R$ such that 
\begin{displaymath} \frac{|\psi(y,t) - \psi(x,t) - A_{p}(y - x)|}{|x - y|} = \epsilon_{p}(|x - y|), \end{displaymath}
where $\epsilon_{p}(r) \to 0$ as $r \to 0$. Combined with (a suitable corollary of) Theorem \ref{dorronsoro}, the "horizontal differentiability" above can be upgraded to "full" differentiability in the following sense:
 \begin{thm}\label{rademacher} Let $\psi \colon \R^{n} \to \R$ be parabolic Lipschitz. Then, for Lebesgue almost every $p = (x,t) \in \R^{n - 1} \times \R$, there exists a linear map $A_{p} \colon \R^{n - 1} \to \R$ such that
 \begin{displaymath} \frac{|\psi(q) - \psi(p) - A_{p}(y - x)|}{d(p,q)} = \epsilon_{p}(d(p,q)), \end{displaymath}
 where $\epsilon_{p}(r) \to 0$ as $r \to 0$.
 \end{thm}
 In other words, the linear map determined by the horizontal gradient automatically approximates $\psi$ also in the vertical direction, at least almost everywhere. This is an analogue of Rademacher's theorem for parabolic Lipschitz functions.

 \subsection{Some auxiliary coefficients} The proof of Theorem \ref{dorronsoro} follows the same idea as the proof of Theorem \ref{main}. Again, one introduces a pair of auxiliary (and "lower dimensional") coefficients, see below, for which, individually, Theorem \ref{dorronsoro} is straightforward. Then, as before, the main trick is to recombine the information from the auxiliary coefficients to say something useful about the numbers $\beta_{2}(Q)$.
 
I now introduce the auxiliary coefficients, defined for continuous functions $\psi \colon \R^{n} \to \R$, and parabolic (not necessarily dyadic) boxes $Q = I_{1} \times I_{2} \subset \R^{n - 1} \times \R$. First define the \emph{horizontal affinity coefficient}
\begin{equation}\label{HAffCoeff} A(Q) := \left( \fint_{I_{2}} \left[ \inf_{A_{t}} \fint_{I_{1}} \left( \frac{|\psi(x,t) - A_{Q,t}(x)|}{\diam(I_{1})} \right)^{2} \, dx \right] \, dt \right)^{1/2}, \end{equation}
where the $\inf$ runs over all affine maps $A_{t} \colon \R^{n - 1} \to \R^{n}$. The reader may note that $A(Q)$ could be re-written as an $L^{2}$ average over the Euclidean $\beta_{2}(I_{1})$ coefficients associated to the maps $x \mapsto \psi(x,t)$. Next, define the \emph{vertical oscillation coefficient}
\begin{equation}\label{VOscCoeff} \osc(Q) := \left( \fint_{I_{1}} \left[ \inf_{c_{x}} \fint_{I_{2}} \frac{|\psi(x,t) - c|^{2}}{|I_{2}|} \, dt \right] \, dx \right)^{1/2}, \end{equation}
where the $\inf$ runs over constants $c_{x} \in \R$. This coefficient measures, how much better than $\tfrac{1}{2}$-H\"older the the map $\psi$ is in the $t$-variable, inside the box $Q$.

I also define the $L$-restricted version $A^{L}(Q)$, where the "$\inf_{A_{t}}$" only runs over affine maps $A_{t} \colon \R^{n - 1} \to \R$ with Lipschitz constant bounded from above by $L \geq 1$. Now, the proof of Theorem \ref{dorronsoro} proceeds as follows: one first establishes suitable estimates for the horizontal affinity and vertical oscillation coefficients separately, using Dorronsoro's theorem in $\R^{n}$ and the condition \eqref{DtCarleson}, respectively. Finally, one verifies that
\begin{displaymath} \beta_{2}(Q)^{2} \lesssim A(Q)^{2} + \osc(Q)^{2} \end{displaymath} 
for all parabolic boxes $Q \subset \R^{n}$. Combining these pieces gives Theorem \ref{dorronsoro}.

\subsection{Estimates for the auxiliary coefficients} I start with the vertical oscillation coefficients:

\begin{proposition}\label{oscCarleson} If $\psi \colon \R^{n} \to \R$ is continuous and satisfies \eqref{DtCarleson}, then, for any $C \geq 1$,
\begin{displaymath} \sum_{Q \subset Q_{0}} \osc^{2}(CQ)|Q| \lesssim_{C} C_{L}|Q_{0}|, \qquad Q_{0} \in \calD. \end{displaymath}
Here $C_{L}$ is the constant in \eqref{DtCarleson}.
\end{proposition}

\begin{proof} Fix $C \geq 1$. Then, fix $x \in \R^{n - 1}$ and consider the map $\psi_{x}(t) = \psi(x,t)$. Let $Q_{0} = I_{1}^{0} \times I_{2}^{0} \in \calD$ with
\begin{displaymath} \ell(I_{1}^{0}) = 2^{-j_{0}} \quad \text{and} \quad |I_{2}^{0}| = 2^{-2j_{0}}. \end{displaymath}
 For dyadic intervals $I \subset I_{2}^{0}$, let $\hat{I} = I + 2C|I| = \{t \in \R : t - 2C|I| \in I\}$, and write
\begin{displaymath} c_{I} := \fint_{\hat{I}} \psi(s) \, ds. \end{displaymath}
Estimate as follows:
\begin{equation}\label{form1a} \inf_{c} \int_{CI} \frac{|\psi_{x}(t) - c|^{2}}{|CI|} \, dt \leq \int_{CI} \frac{|\psi_{x}(t) - c_{I}|^{2}}{|CI|} \, dt \leq \fint_{\hat{I}} \int_{CI} \frac{|\psi_{x}(t) - \psi_{x}(s)|^{2}}{|CI|} \, dt \, ds. \end{equation} 
Next, note that whenever $t \in CI$ and $s \in \hat{I}$, then 
\begin{displaymath} c_{1}|I| \leq r := s - t \leq C_{1}|I| \end{displaymath}
for some $c_{1},C_{2} \sim_{C} 1$. Hence, after the change of variable $s \mapsto t + r$, and noting that $|\hat{I}| = |I|$, one gets
\begin{displaymath} \eqref{form1a} \lesssim_{C} \int_{c_{1}|I|}^{C_{1}|I|} \int_{CI} \frac{|\psi_{x}(t) - \psi_{x}(t + r)|^{2}}{|I|^{2}} \, dt \, dr. \end{displaymath}
For $j \geq j_{0}$ fixed, it follows that
\begin{align} \mathop{\sum_{I \in \calD^{1}_{2j}}}_{I \subset I_{2}^{0}} \inf_{c_{I}} \int_{CI} \frac{|\psi_{x}(t) - c_{I}|^{2}}{|CI|} \, dt & \lesssim_{C} \int_{c_{1}2^{-2j}}^{C_{1}2^{-2j}} \mathop{\sum_{I \in \calD^{1}_{2j}}}_{I \subset I_{2}^{0}} \int_{CI} \frac{|\psi_{x}(t) - \psi_{x}(t + r)|^{2}}{|I|^{2}} \, dt \, dr \notag\\
&\label{form2a} \lesssim_{C} \int_{c_{1}2^{-2j}}^{C_{1}2^{-2j}} \int_{CI_{2}^{0}} \frac{|\psi_{x}(t) - \psi_{x}(t + r)|^{2}}{r^{2}} \, dt \, dr. \end{align}
To complete the proof, start by writing
\begin{align*} \sum_{Q \subset Q_{0}} \osc(CQ)^{2}|Q| & = \mathop{\sum_{I_{1} \times I_{2} \in \calD}}_{I_{1} \times I_{2} \subset I_{1}^{0} \times I_{2}^{0}} \left( \fint_{CI_{1}} \left[ \inf_{c_{Q,x}} \fint_{CI_{2}} \frac{|\psi(x,t) - c_{Q,x}|^{2}}{|CI_{2}|} \, dt \right]  \, dx \right) |Q|\\
& \sim_{C} \sum_{j \geq j_{0}} \mathop{\sum_{I_{1} \in \calD^{n - 1}_{j}}}_{I_{1} \subset I_{1}^{0}} \mathop{\sum_{I_{2} \in \calD^{1}_{2j}}}_{I_{2} \subset I_{2}^{0}} \left( \int_{CI_{1}} \left[ \inf_{c_{Q,x}} \int_{CI_{2}} \frac{|\psi(x,t) - c_{Q,x}|^{2}}{|CI_{2}|} \, dt \right]  \, dx \right)\\
& = \sum_{j \geq j_{0}} \int_{CI_{1}^{0}} \mathop{\sum_{I_{1} \in \calD^{n - 1}_{j}}}_{x \in CI_{1}} \mathop{\sum_{I_{2} \in \calD^{1}_{2j}}}_{I_{2} \subset I_{2}^{0}} \left[ \inf_{c_{Q,x}} \int_{CI_{2}} \frac{|\psi_{x}(t) - c_{Q,x}|^{2}}{|CI_{2}|} \, dt \right] \, dx. \end{align*}
For $j \geq j_{0}$ and $x \in CI_{1}^{0}$ fixed, the innermost sum can next be estimated by applying \eqref{form2a}. Since moreover $\card\{I_{1} \in \calD_{j}^{n - 1} : x \in CI_{1}\} \lesssim_{C} 1$ for any $j \geq j_{0}$ fixed, one obtains
\begin{align*} \sum_{Q \subset Q_{0}} \osc(CQ)^{2}|Q| & \lesssim_{C} \sum_{j \geq j_{0}} \int_{CI_{1}^{0}} \int_{c_{1}2^{-2j}}^{C_{1}2^{-2j}} \int_{CI_{0}^{2}} \frac{|\psi_{x}(t) - \psi_{x}(t + r)|^{2}}{r^{2}} \, dt \, dr \, dx\\
& \sim_{C} \int_{CI_{1}^{0}} \int_{0}^{C_{1}|I_{0}^{2}|} \int_{CI_{2}^{0}} \frac{|\psi_{x}(t) - \psi_{x}(t + r)|^{2}}{r^{2}} \, dt \, dr \, dx\\
& \leq \int_{CI_{1}^{0}} \int_{CI_{2}^{0}} \int_{C_{2}I_{2}^{0}} \frac{|\psi(x,t) - \psi(x,s)|^{2}}{|t - s|^{2}} \, ds \, dt \, dx \lesssim_{C} C_{L}|Q_{0}|, \end{align*}
using \eqref{DtCarleson} in the last inequality; on the last line also $C_{2} \sim C + C_{2}$ is a constant. This completes the proof. \end{proof}

Next follows a similar estimate for the horizontal affinity coefficients:

\begin{proposition}\label{affCarleson} If $\psi \colon \R^{n} \to \R$ is continuous and satisfies \eqref{horizontalLipschitz} for some $L \geq 1$, then, for any $C \geq 1$,
\begin{displaymath} \mathop{\sum_{Q \in \calD}}_{Q \subset Q_{0}} A^{L}(CQ)^{2}|Q| \lesssim_{C} L|Q_{0}|, \qquad Q_{0} \in \calD. \end{displaymath}
\end{proposition}

\begin{proof} As in the previous proof, write $Q_{0} =: I_{1}^{0} \times I_{2}^{0} \in \calD$, and
\begin{displaymath} \ell(I_{1}^{0}) =: 2^{-j_{0}} \quad \text{and} \quad |I_{2}^{0}| =: 2^{-2j_{0}}. \end{displaymath}
The proposition is an application of Dorronsoro's theorem. If $t \in \R$ is fixed, then the map $\psi^{t} \colon \R^{n - 1} \to \R$ defined by $\psi^{t}(x) := \psi(x,t)$ is $L$-Lipschitz, uniformly in $t \in \R$, so Theorem \ref{main} implies that
\begin{equation}\label{dor} \mathop{\sum_{I \in \calD^{n - 1}}}_{I \subset I_{1}^{0}} \left( \inf_{A_{I}} \fint_{CI} \left( \frac{|\psi^{t}(x) - A_{I}(x)|}{\diam(CI)} \right)^{2} \, dx \right)|I| \lesssim_{C} L|I_{1}^{0}| \end{equation}
for any $C \geq 1$. Here $\inf$ runs over $L$-Lipschitz affine functions $A_{I} \colon \R^{n - 1} \to \R$: Theorem \ref{main} does not explicitly mention that the affine maps can be taken $L$-Lipschitz, but one can see this in various ways. For example, one can recall how the map $A$ was constructed in Section \ref{conclusion}, and see that its Lipschitz constant is no larger then the Lipschitz constant of $\psi^{t}$. For another proof, see \cite[Section 7.3]{MR3512428} where the form of the affine approximation is explicit, see \cite[p. 645]{MR3512428}, and $\|\nabla_{x} A_{I}\|_{\infty} \leq \|\nabla_{x} \psi^{t}\|_{\infty}$.

Now, to prove prove the proposition, first expand as follows:
\begin{align*} \mathop{\sum_{Q \in \calD}}_{Q \subset Q_{0}} A^{L}(CQ)^{2}|Q| & = \mathop{\sum_{I_{1} \times I_{2} \in \calD}}_{I_{1} \times I_{2} \subset I_{1}^{0} \times I_{2}^{0}} \left( \fint_{CI_{2}} \left[ \inf_{A_{Q,t}} \fint_{CI_{1}} \left( \frac{|\psi^{t}(x) - A_{Q,t}(x)|}{\diam(CI_{1})} \right)^{2} \, dx \right] \, dt \right) |Q| \\
& \sim_{C} \sum_{j \geq j_{0}} \mathop{\sum_{I_{2} \in \calD_{2j}^{1}}}_{I_{2} \subset I_{2}^{0}} \mathop{\sum_{I_{1} \in \calD^{n - 1}_{j}}}_{I_{1} \subset I_{1}^{0}} \left( \int_{CI_{2}} \left[ \inf_{A_{Q,t}} \int_{CI_{1}} \left( \frac{|\psi^{t}(x) - A_{Q,t}(x)|}{\diam(CI_{1})} \right)^{2} \, dx \right] \, dt \right) \\
& \leq \sum_{j \geq j_{0}} \int_{CI_{2}^{0}} \mathop{\sum_{I_{2} \in \calD^{1}_{2j}}}_{t \in CI_{2}} \bigg( \mathop{\sum_{I_{1} \in \calD^{n - 1}_{j}}}_{I_{1} \subset I_{1}^{0}} \left[ \inf_{A_{Q,t}} \int_{CI_{1}} \left( \frac{|\psi^{t}(x) - A_{Q,t}(x)|}{\diam(CI_{1})} \right)^{2} \, dx \right] \, \bigg) dt. \end{align*}
Note that the expression in brackets is independent of the choice of $I_{2}$ in the middle summation. Hence, writing $N_{j}(t) := \card\{I_{2} \in \calD_{2j}^{1} : t \in CI_{2}\}$, one notes that $N_{j}(t) \lesssim_{C} 1$ uniformly for $j \geq j_{0}$ and $t \in \R$ fixed, and the previous display can be continued as follows:
\begin{align*}  \ldots = & \sum_{j \geq j_{0}} \int_{CI_{2}^{0}} N_{j}(t) \mathop{\sum_{I_{1} \in \calD^{n - 1}_{j}}}_{I_{1} \subset I_{1}^{0}} \left[ \inf_{A_{Q,t}} \int_{CI_{1}} \left(\frac{|\psi^{t}(x) - A_{Q,t}(x)|}{\diam(CI_{1})} \right)^{2} \, dx \right] \, dt\\
& \lesssim_{C} \sum_{j \geq j_{0}} \int_{CI_{2}^{0}} \mathop{\sum_{I_{1} \in \calD^{n - 1}_{j}}}_{I_{1} \subset I_{1}^{0}} \left[ \inf_{A_{Q,t}} \int_{CI_{1}} \left( \frac{|\psi^{t}(x) - A_{Q,t}(x)|}{\diam(CI_{1})} \right)^{2} \, dx \right] \, dt\\
& = \int_{CI_{2}^{0}} \mathop{\sum_{I_{1} \in \calD_{j}^{n - 1}}}_{I_{1} \subset I_{1}^{0}} \left[ \inf_{A_{Q,t}} \int_{CI_{1}} \left( \frac{|\psi^{t}(x) - A_{Q,t}(x)|}{\diam(CI_{1})} \right)^{2} \, dx \right] \, dt \stackrel{\eqref{dor}}{\lesssim_{C}} L\int_{CI_{2}^{0}} |I_{1}^{0}| \, dt \sim_{C} L|Q_{0}|. \end{align*} 
This completes the proof. \end{proof}

\subsection{Proof of the Dorronsoro estimate}
This section contains a proof of the inequality
\begin{equation}\label{form10} \beta(Q)^{2} \lesssim A(Q)^{2} + \osc(Q)^{2} \end{equation}
for any parabolic box $Q \subset \R^{n}$, and the conclusion of the proof of Theorem \ref{dorronsoro}. The "parabolicity" of the rectangle $Q$ plays no role in \eqref{form10}, so I formulate the following proposition:

\begin{proposition}\label{AffApprox} Let $I_{1} \subset \R^{n - 1}$ be a cube, and let $I_{2} \subset \R$ be an interval. Assume that $f \colon Q := I_{1} \times I_{2} \to \R$ is a bounded continuous function. For $x \in \R^{n - 1}$ and $t \in \R$ fixed, write 
\begin{displaymath} f^{t}(x) = f(x,y) \quad \text{and} \quad f_{x}(t) = f(x,y). \end{displaymath}
Define 
\begin{equation}\label{HAff} \beta := \fint_{I_{2}} \inf_{A_{t}} \fint_{I_{1}} |f^{t}(x) - A_{t}(x)|^{2} \, dx \, dt, \end{equation}
and
\begin{displaymath} \beta := \fint_{I_{1}} \inf_{c_{x}} \fint_{I_{2}} |f_{x}(t) - c_{x}|^{2} \, dt \, dx, \end{displaymath} 
where $A_{t}$ runs over affine maps $\R^{n - 1} \to \R$, and $c_{x}$ runs over $\R$. Then, there exists an affine map $A \colon \R^{n - 1} \to \R$ such that
\begin{displaymath} \fint_{Q} |f(x,t) - A(x)|^{2} \, dx \, dt \lesssim \beta + \beta. \end{displaymath}
If moreover the maps $A_{t}$ are $L$-Lipschitz, then also $A$ is $L$-Lipschitz.
\end{proposition}

\begin{proof} For each $t \in I_{2}$, select an affine map $A_{t}(x) = a_{t} \cdot x + b_{t}$ with $a_{t} \in \R^{n - 1}$ and $b_{t} \in \R$ minimising 
\begin{displaymath} \fint_{I_{1}} |f^{t}(x) - A_{t}(x)|^{2} \, dx. \end{displaymath}
The minimiser exists as an $L^{2}$-projection, and the maps $t \mapsto a_{t}$ and $t \mapsto b_{t}$ are continuous and bounded: to see this, note that $t \mapsto f^{t}$ is a continuous map from $I_{2}$ to $L^{2}(I_{1})$, and the $L^{2}$-projection is further a continuous map from $L^{2}(I_{1})$ to the finite-dimensional subspace of affine functions, where $|a_{t}| + |b_{t}| \sim_{I_{1}} \|A_{t}\|_{2}$. These claims continue to hold if $A_{t}$ is restricted to being $L$-Lipschitz (noting that the property of being $L$-Lipschitz defines a convex set of affine maps, so the projection is well-defined). Further,
\begin{displaymath} \fint_{I_{2}} \fint_{I_{1}} |f^{t}(x) - A_{t}(x)|^{2} \, dx \, dt = \beta. \end{displaymath} 
Similarly, choose constants $c_{x} \in \R$, $x \in I_{1}$, such that
\begin{displaymath} \fint_{I_{1}} \fint_{I_{2}} |f_{x}(t) - c_{x}|^{2} \, dt \, dx = \beta. \end{displaymath}
Now, there is only one reasonable way to define the affine map $A \colon \R^{n - 1} \to \R$, namely
\begin{displaymath} A(x) := \fint_{I_{2}} A_{t}(x) \, dt. \end{displaymath}
This formula indeed produces an affine map: in fact $A(x) = a \cdot x + b$, where
\begin{displaymath} a = \fint_{I_{2}} a_{t} \, dt \in \R^{n - 1} \quad \text{and} \quad b = \fint_{I_{2}} b_{t} \, dt. \end{displaymath}
Note that if every $A_{t}$ is $L$-Lipschitz, then $|a_{t}| \leq L$, hence also $|a| \leq L$, and consequently also $A$ is $L$-Lipschitz. Recall that the mean is the best $L^{2}$ approximation: 
\begin{displaymath} \E [|X - \E[X]|^{2}] = \inf_{c \in \R} \E [|X - c|^{2}].  \end{displaymath}
Applied to the random variable $X : t \mapsto A_{t}(x)$ for a fixed $x \in I_{1}$, this gives
\begin{displaymath} \fint_{I_{2}} |A_{t}(x) - A(x)|^{2} \, dy = \fint_{I_{2}} \left| A_{t}(x) - \fint_{I_{2}} A_{t}(x) \, dt \right|^{2} \, dt \leq \fint_{I_{2}} |A_{t}(x) - c_{x}|^{2} \, dt. \end{displaymath}
Using this and \eqref{HAff}, one first finds that
\begin{align*} \fint_{Q} |f(x,t) - A(x)|^{2} \, dx \, dt & \lesssim \fint_{I_{2}} \fint_{I_{1}} |f^{t}(x) - A_{t}(x)|^{2} \, dx \, dt\\
& \quad + \fint_{I_{1}} \fint_{I_{2}} |A_{t}(x) - A(x)|^{2} \, dt \, dx\\
& \leq \beta + \fint_{I_{1}}\fint_{I_{2}} |A_{t}(x) - c_{x}|^{2} \, dt \, dx.  \end{align*} 
The second term on the right also has the correct bound:
\begin{align*} \fint_{I_{1}}\fint_{I_{2}} |A_{t}(x) - c_{x}|^{2} \, dt \, dx & \lesssim \fint_{I_{1}} \fint_{I_{2}} |f^{t}(x)  - A_{t}(x)|^{2} \, dt \, dx \\
& + \fint_{I_{2}} \fint_{I_{1}} |f_{x}(t) - c_{x}|^{2} \, dt \, dx = \beta + \beta,  \end{align*} 
recalling the choices of $A_{t}$ and $c_{x}$. The proof is complete. \end{proof}

\begin{cor} Let $\psi \colon \R^{n} \to \R$ be continuous. Then,
\begin{equation}\label{form4a} \beta_{2}(Q)^{2} \lesssim A(Q)^{2} + \osc(Q)^{2} \end{equation}
for any parabolic box $Q \subset \R^{n}$. The same inequality holds if $A(Q)$ and $\beta_{2}(Q)$ are replaced by $A^{L}(Q)$ and $\beta_{2}^{L}(Q)$. \end{cor} 

\begin{proof} Write $Q = I_{1} \times I_{2}$, and note that $\ell(I_{1})^{2} \sim |I_{2}| \sim \diam(Q)^{2}$. This means that the normalisation in each of the three coefficients \eqref{AffCoeff} and \eqref{HAffCoeff}-\eqref{VOscCoeff} is roughly the same and can be factored out. After this observation, \eqref{form4a} is a direct corollary of Proposition \ref{AffApprox} applied to $f = \psi$ and $Q = I_{1} \times I_{2}$. The proof of the inequality with the coefficients $A^{L}(Q)$ and $\beta_{2}^{L}(Q)$ is the same, taking into account the last part of Proposition \ref{AffApprox}. \end{proof}

A combination of the previous arguments now yields Theorem \ref{dorronsoro}

\begin{proof}[Proof of Theorem \ref{dorronsoro}] For $\psi \colon \R^{n} \to \R$ parabolic Lipschitz, and for $Q_{0} \in \calD$, estimate as follows:
\begin{displaymath} \mathop{\sum_{Q \in \calD}}_{Q \subset Q_{0}} \beta_{2}(CQ)^{2}|Q| \lesssim \mathop{\sum_{Q \in \calD}}_{Q \subset Q_{0}} A(CQ)^{2}|Q| + \mathop{\sum_{Q \in \calD}}_{Q \subset Q_{0}} \osc(CQ)^{2}|Q| \lesssim_{C} |Q_{0}|, \end{displaymath}
using \eqref{form4a} in the first inequality, and Propositions \ref{oscCarleson} and \ref{affCarleson} in the second. If $\psi$ is parabolic $L$-Lipschitz, the same argument works for the coefficients $\beta_{2}^{L}(Q)$ and $A^{L}(Q)$. The proof of Theorem \ref{dorronsoro} is complete.
 \end{proof}
 
 \subsection{Rademacher's theorem} Define the natural $L^{\infty}$ variant of the coefficients $\beta_{p}(Q)$ as follows: if $\psi \colon \R^{n} \to \R$ is continuous, set
 \begin{displaymath} \beta_{\infty}(Q) := \inf_{A} \sup_{(x,t) \in Q} \frac{|\psi(x,t) - A(x)|}{\diam(Q)}, \end{displaymath}
 where $A$ again ranges over affine maps $A \colon \R^{n - 1} \to \R$. As before, define $\beta_{\infty}^{L}(Q)$, where the affine maps are restricted to being $L$-Lipschitz. In this section, I recall that parabolic Lipschitz functions are Lipschitz-continuous in the metric $d$ (as noted in \eqref{lipschitz}). Consequently, if $\psi$ is parabolic Lipschitz, one has 
 \begin{equation}\label{betaLeq1} \beta^{L}_{\infty}(Q) \lesssim 1, \qquad Q \subset \R^{n}, \: L \geq 1. \end{equation}
 
\begin{lemma}\label{betaInftyVsBeta2} Let $L \geq 1$, and let $\psi \colon \R^{n} \to \R$ be $L$-Lipschitz in the metric $d$. Then,
\begin{displaymath} \beta^{L}_{\infty}(Q) \lesssim_{L} \beta^{L}_{2}(2Q)^{2/(n + 3)} \end{displaymath} 
for all parabolic boxes $Q \subset \R^{n}$.
\end{lemma} 

\begin{proof} This argument is \cite[(5.4)]{DS1} translated to the parabolic setting. Let $A \colon \R^{n - 1} \to \R$ be an affine $L$-Lipschitz function almost minimising $\beta := \beta^{L}_{2}(2Q)$:
\begin{displaymath} \fint_{2Q} \left( \frac{|\psi(y,s) - A(y)|}{\diam(2Q)} \right)^{2} \, dy \, ds \leq 2\beta^{2}. \end{displaymath}
By \eqref{betaLeq1} the claim of the lemma is clear if $\beta \sim_{L} 1$, so one may assume that $\beta$ is small. Now, consider a point $p = (x,t) \in Q$, and assume to reach a contradiction that
\begin{displaymath} |\psi(x,t) - A(x)| \geq 2C\beta^{2/(n + 3)}\diam(Q) \end{displaymath}
for some large constant $C \geq 1$. Since both $\psi$ and $A$ are $L$-Lipschitz, one infers that there exists a $d$-ball $B \subset \R^{n}$ with $\diam(B) \sim_{L} \beta^{2/(n + 3)}\diam(Q)$ such that
\begin{displaymath} |\psi(y,s) - A(y)| \geq C\beta^{2/(n + 3)}\diam(Q), \qquad (y,s) \in B. \end{displaymath} 
Assuming that $\beta$ is sufficiently small, and recalling that $p \in Q$, one has $B \subset 2Q$. It follows that

\begin{align*} \beta^{2(n + 1)/(n + 3)}\diam(Q)^{n + 1} \sim_{L} \mathcal{L}^{n}(B) & \lesssim \frac{1}{C^{2}\beta_{2}^{4/(n + 3)}} \int_{B} \left( \frac{|\psi(y,s) - A(y)|}{\diam(Q)} \right)^{2} \, dy \, ds\\
& \lesssim \frac{\diam(Q)^{n + 1}}{C^{2}\beta^{4/(n + 3)}} \fint_{2Q} \left( \frac{|\psi(y,s) - A(y)|}{\diam(2Q)} \right)^{2} \, dy \, ds\\
& \lesssim \frac{\beta^{2}\diam(Q)^{n + 1}}{C^{2}\beta^{4/(n + 3)}}. \end{align*} 
Rearranging terms leads to $1 \lesssim_{L} 1/C^{2}$, which gives the desired contradiction for $C \geq 1$ large enough.  \end{proof}

 \begin{cor}\label{betaInftyCarleson} If $\psi \colon \R^{n} \to \R$ is parabolic Lipschitz, then, for any $C \geq 1$,  
  \begin{displaymath} \mathop{\sum_{Q \in \calD}}_{Q \subset Q_{0}} \beta_{\infty}(CQ)^{n + 3}|Q| \lesssim_{C} |Q_{0}|, \qquad Q_{0} \in \calD. \end{displaymath} 
 \end{cor}
 
 \begin{proof} Let $L \geq 1$ be a constant such that $\psi$ is parabolic $L$-Lipschitz. Then, combining Theorem \ref{dorronsoro} and Proposition \ref{betaInftyVsBeta2} gives
 \begin{displaymath} \mathop{\sum_{Q \in \calD}}_{Q \subset Q_{0}} \beta_{\infty}^{L}(CQ)^{n + 3}|Q| \lesssim_{L} \mathop{\sum_{Q \in \calD}}_{Q \subset Q_{0}} \beta_{2}^{L}(2CQ)^{2}|Q| \lesssim_{C} |Q_{0}|, \qquad Q_{0} \in \calD, \end{displaymath} 
 and the claim now follows from $\beta_{\infty}(CQ) \leq \beta_{\infty}^{L}(CQ)$.  \end{proof}
 
Proposition \ref{betaInftyCarleson} implies that if $\psi \colon \R^{n} \to \R$ is parabolic Lipschitz, then
\begin{displaymath} \mathop{\sum_{p \in Q \in \calD}}_{\diam(Q) \leq 1} \beta_{\infty}(2Q)^{n + 3} < \infty \end{displaymath}
for Lebesgue almost every $p \in \R^{n}$. For these $p \in \R^{n}$, one has $\beta_{\infty}(2Q) \to 0$ as $Q \to p$. Theorem \ref{rademacher} is an easy consequence of this fact, combined with Rademacher's theorem for Euclidean Lipschitz functions. Recall the statement of Theorem \ref{rademacher}:
 \begin{thm} Let $\psi \colon \R^{n} \to \R$ be parabolic Lipschitz. Then, for Lebesgue almost every $p = (x,t) \in \R^{n - 1} \times \R$, there exists a linear map $A_{p} \colon \R^{n - 1} \to \R$ such that
 \begin{equation}\label{differentiability} \frac{|\psi(q) - \psi(p) - A_{p}(y - x)|}{d(p,q)} = \epsilon_{p}(d(p,q)), \end{equation}
 where $\epsilon_{p}(r) \to 0$ as $r \to 0$.
 \end{thm}
 
 \begin{proof} By the Euclidean Rademacher theorem (and Fubini's theorem), for Lebesgue almost every $p = (x,t) \in \R^{n}$, there exists a linear map $A_{p} \colon \R^{n - 1} \to \R$ such that
 \begin{equation}\label{form5a} \frac{|\psi(y,t) - \psi(p) - A_{p}(y - x)|}{|x - y|} = \tilde{\epsilon}_{p}(|y - x|), \end{equation} 
where $\tilde{\epsilon}_{p}(r) \to 0$ as $r \to 0$. It remains to show that \eqref{differentiability} holds with the linear map $A_{p}$ almost everywhere. According to Corollary \ref{betaInftyCarleson}, one has
\begin{equation}\label{form6a} \lim_{Q \to p} \beta_{\infty}(2Q) = 0 \end{equation} 
for Lebesgue almost every $p \in \R^{n}$. Fix a point $p \in \R^{n}$ such that \eqref{form5a}-\eqref{form6a} hold: the claim is that also \eqref{differentiability} holds. 

One first needs to relate the best approximating affine maps from the definition of $\beta_{\infty}(2Q)$ to the fixed map $A_{p}$, see \eqref{form8a} below for where this is heading. For $Q \in \calD$ fixed with $p \in Q$, let $A^{Q} =  A_{Q} + b \colon \R^{n - 1} \to \R$ be an affine map almost minimising $\beta_{\infty}(2Q)$, where $A_{Q}$ is linear, $b \in \R$, and
\begin{displaymath} \sup_{(y,s) \in 2Q} \frac{|\psi(y,s) - A^{Q}(y)|}{\diam(2Q)} \leq 2\beta_{\infty}(2Q). \end{displaymath}
Since also $p \in Q \subset 2Q$, one has
\begin{equation}\label{form9a} |\psi(y,t) - \psi(p) - A_{Q}(y - x)| \lesssim \beta_{\infty}(2Q)\diam(Q), \qquad (y,t) \in 2Q. \end{equation} 
Combining this with \eqref{form5a},
\begin{equation}\label{form8a} |A_{Q}(y - x) - A_{p}(y - x)| \lesssim \delta(Q)\diam(Q), \qquad y \in 2I_{1}, \end{equation}
with $\delta(Q) := \beta_{\infty}(2Q) + \tilde{\epsilon}(\diam(2Q))$, and where $2I_{1}$ stands for the base of the cube $2Q = 2I_{1} \times 2I_{2} \subset \R^{n - 1} \times \R$.

Now, to complete the proof of \eqref{differentiability}, note that for any $q = (y,s) \neq p$, there exists a cube $Q \in \calD$ with 
\begin{equation}\label{form11a} p \in Q, \quad q \in 2Q, \quad \text{and} \quad \diam(Q) \sim d(p,q). \end{equation}
Define
\begin{displaymath} \epsilon_{p}(r) := \sup \{\delta(Q) : p \in Q \in \calD \text{ and } \diam(Q) \sim r\}, \end{displaymath} 
where the implicit constant is the same as in \eqref{form11a}. Since $\delta(Q) \to 0$ as $Q \to p$, also $\epsilon_{p}(r) \to 0$ as $r \to 0$. Finally, for $q = (y,s) \neq p$, fix $Q \in \calD$ as in \eqref{form11a}, and estimate as follows, using \eqref{form9a}-\eqref{form8a}:
\begin{align*} |\psi(q) - \psi(p) - A_{p}(y - x)| & \leq |\psi(q) - \psi(p) - A_{Q}(y - x)| + |A_{Q}(y - x) - A_{p}(y - x)|\\
& \lesssim [\beta_{\infty}(2Q) + \delta(Q)]d(p,q) \lesssim \epsilon_{p}(d(p,q))d(p,q). \end{align*} 
This completes the proof of \eqref{differentiability}. \end{proof}

\bibliographystyle{plain}
\bibliography{references2}

\end{document}